\documentclass[12pt,twoside]{amsart}
\usepackage{a4wide, amsmath,amsbsy,amsfonts,amssymb,mathabx,
stmaryrd,amsthm,mathrsfs,graphicx, amscd, tikz-cd, xcolor, cancel}
\usepackage[all]{xy}
\usepackage[normalem]{ulem}
\usetikzlibrary{matrix,arrows,decorations.pathmorphing}
\usepackage[active]{srcltx}
\usepackage[
	hypertexnames=false,
	hyperindex,
	pagebackref,
	pdftex,
	breaklinks=true,
	bookmarks=false,
	colorlinks,
	linkcolor=blue,
	citecolor=red,
	urlcolor=red,
]{hyperref}
\usepackage[mathscr]{eucal}
\DeclareMathSymbol{\Alpha}{\mathalpha}{operators}{"41}
\DeclareMathSymbol{\Beta}{\mathalpha}{operators}{"42}
\DeclareMathSymbol{\Epsilon}{\mathalpha}{operators}{"45}
\DeclareMathSymbol{\Zeta}{\mathalpha}{operators}{"5A}
\DeclareMathSymbol{\Eta}{\mathalpha}{operators}{"48}
\DeclareMathSymbol{\Iota}{\mathalpha}{operators}{"49}
\DeclareMathSymbol{\Kappa}{\mathalpha}{operators}{"4B}
\DeclareMathSymbol{\Mu}{\mathalpha}{operators}{"4D}
\DeclareMathSymbol{\Nu}{\mathalpha}{operators}{"4E}
\DeclareMathSymbol{\Omicron}{\mathalpha}{operators}{"4F}
\DeclareMathSymbol{\Rho}{\mathalpha}{operators}{"50}
\DeclareMathSymbol{\Tau}{\mathalpha}{operators}{"54}
\DeclareMathSymbol{\Chi}{\mathalpha}{operators}{"58}
\DeclareMathSymbol{\omicron}{\mathord}{letters}{"6F}

\newcommand\la{\langle}
\newcommand\ra{\rangle}

\newcommand\Z{{\mathbb Z}}
\newcommand\ZZ{\widehat{\mathbb Z}}

\newcommand\F{{\mathbb F}}

\newcommand\cR{{\mathcal R}}
\newcommand\cP{{\mathscr P}}
\newcommand\Ra{\Rightarrow}
\newcommand\SL{\operatorname{SL}}

\newcommand\Sp{\operatorname{Sp}}

\newcommand\hookra{\hookrightarrow}

\renewcommand{\hom}{\operatorname{Hom}}
 
\newcommand\sr{\stackrel}

\newcommand\ssm{\smallsetminus}
\newcommand\ol{\overline}

\newcommand\cC{{\mathscr C}}

\newcommand\cS{{\mathscr S}}

\newcommand\cT{{\mathcal T}}
\newcommand\cI{{\mathcal I}}

\newcommand\G{\Gamma}
\newcommand\U{\Upsilon}

\newcommand\hG{{\widehat{G}}}

\newcommand\ld{\lambda}
\newcommand\Ld{\Lambda}
\newcommand\wh{\widehat}

\newcommand\td{\tilde}

\def\co{\colon\thinspace}

\newcommand\SO{\operatorname{SO}}

\newcommand\cd{\operatorname{cd}}
\newcommand\vcd{\operatorname{vcd}}

\def\wC{{\widehat{\cC}}}
\def\u{{\upsilon}}
\def\dS{{\partial S}}

\newtheorem{innercustomthm}{Theorem}
\newenvironment{customtheorem}[1]
  {\renewcommand\theinnercustomthm{#1}\innercustomthm}
  {\endinnercustomthm}
  
\newtheorem{innercustomcor}{Corollary}
\newenvironment{customcorollary}[1]
  {\renewcommand\theinnercustomcor{#1}\innercustomcor}
  {\endinnercustomcor}

\newtheorem{theorem}{Theorem}[section]

\newtheorem{proposition}[theorem]{Proposition}

\newtheorem{lemma}[theorem]{Lemma}

\theoremstyle{definition}
\newtheorem{definition}[theorem]{Definition}   
\newtheorem{remark}[theorem]{Remark}

\begin{document} 

\title{Finite subgroups of the profinite completion of good groups}

\author{Marco Boggi}
\address{UFF - Instituto de Matem\'atica e Estat\'{\i}stica -
Niter\'oi - RJ 24210-200; Brazil}
\email{marco.boggi@gmail.com}

\author{Pavel Zalesskii}
\address{Departamento de Matem\'atica, Universidade de Bras\'{\i}lia, 
70910-900, Bras\'{\i}lia-DF; Brazil.} 
\email{pz@mat.unb.br}\thanks{The second author was partially supported  by FAPDF and CNPq.}

\begin{abstract}Let $G$ be a residually finite, good group of finite virtual cohomological dimension. 
We prove that  the natural monomorphism $G\hookra\hG$ induces a bijective correspondence 
between conjugacy classes of finite $p$-subgroups of $G$ and those of its profinite completion $\hG$.
Moreover, we prove that the centralizers and normalizers in $\hG$ of finite $p$-subgroups of $G$ 
are the closures of the respective centralizers and normalizers in $G$. 
With somewhat more restrictive hypotheses,  we prove the same results for finite solvable subgroups of $G$. 
In the last section, we give a few applications of this theorem to hyperelliptic mapping class groups and virtually 
compact special toral relatively hyperbolic groups (these include fundamental 
groups  of $3$-orbifolds and of uniform standard arithmetic hyperbolic orbifolds).
\vskip 0.2cm
\noindent AMS Math Classification: 20E18, 20E26, 20F65.

\end{abstract}

\maketitle

\section{Introduction}
Let $G$ be a residually finite group, $\hG$ its profinite completion and $\iota\co G\hookra\hG$ the natural 
monomorphism. The map $\iota$ induces a map $\iota_c$ from the set of conjugacy classes 
of finite subgroups of $G$ to the set of conjugacy classes of finite subgroups of $\hG$. A natural question is then whether this map
is injective, in which case we say that $\iota$ separates conjugacy classes of finite subgroups, and whether
it is surjective, i.e.\ whether every finite subgroup of $\hG$ is conjugated to a finite subgroup of $\iota(G)$
(from now on, we will simply denote by $G$ the image of $\iota$). Of course, the question can be formulated for a residually $\cC$ group $G$ and the
pro-$\cC$ completion $G_\wC$ respectively, where $\cC$ is a class of finite groups closed for subgroups, quotients and extensions. 
The classes of finite $p$-groups or finite solvable groups are of particular interest. As matter of notation, a \emph{$\cC$-group} is a group in $\cC$. 

The problem of whether $\iota_c$ is injective is the restriction of the subgroup conjugacy separability problem, namely of whether two distinct
conjugacy classes of subgroups of $G$ have distinct images in some finite quotient of $G$, to finite subgroups of $G$. 
Examples of subgroup conjugacy separable groups are virtually polycyclic groups (cf.\ \cite{GS}), virtually free groups (cf.\ \cite{CZ-15}) and hyperbolic 
virtually compact special groups (in particular, Fuchsian and  one-relator groups with torsion) (cf.\ \cite{CZ-19}). In particular, for all these groups, the map 
$\iota_c$ is injective.

On the other hand, the problem of whether the map $\iota_c$ is surjective has not been much studied. 
This is true for finitely generated nilpotent groups (since, in this case, torsion elements form a finite normal subgroup), for virtually free groups 
by \cite[Theorem 3.10]{ZM} (cf.\ also \cite[Theorem 1.2]{R}). Kropholler and Wilson \cite[Theorem 2.7]{KW} proved the surjectivity of $\iota_c$ 
for residually finite solvable minimax (in particular, polycyclic) groups and constructed counterexamples for infinitely generated nilpotent groups 
of class $2$ and for finitely generated center-by-metabelian groups.

In general, the answer is negative and it is not difficult to give some examples. For instance, for $n\geq 2$,
there holds $\wh{\Sp_{2n}(\Z)}\cong\Sp_{2n}(\ZZ)$ (cf. \cite{BMS}) and the natural homomorphism
$\iota\co\Sp_{2n}(\Z)\to\Sp_{2n}(\ZZ)$ is injective. However, the center of $\Sp_{2n}(\Z)$ consists of the scalar
matrices $\pm{\mathrm I}$ and is a finite subgroup of order $2$, while the center of $\Sp_{2n}(\ZZ)\cong\prod_p\Sp_{2n}(\Z_p)$ contains elementary 
abelian $2$-subgroups of any order. Thus $\iota_c$ in this case is not surjective. In fact, this is the case for any arithmetic group for which the congruence 
subgroup problem has a positive solution.

The reason for which the linear group $\Sp_{2n}(\Z)$ and, in general, arithmetic groups with a positive solution to the congruence subgroup problem fail 
to satisfy the desired property is that these groups are not \emph{good} (cf.\ \cite[Proposition 5.1 and Remark 5.3]{GJZ}). Let us recall this notion.

In this paper, for a given profinite group $R$ and prime number $p>0$, the coefficients for the cohomology of $R$ will be in the category 
of \emph{discrete  $\F_p[[R]]$-modules}, where $\F_p[[R]]$ is the completed group algebra of $R$ with $\F_p$-coefficients. 
A discrete $\F_p[[R]]$-module is equivalently described as an $\F_p$-vector space, with the discrete topology, endowed with a continuous linear action 
of the profinite group $R$.

\begin{definition}\label{good group}Let $p$ be a prime such that $\Z/p\in\cC$.
We say that a group $G$ is \emph{$p$-good} in $\cC$ if, for every discrete $\F_p[[G_\wC]]$-module $M$, 
the homomorphism induced on cohomology $\iota^\ast\co H^i(G_\wC,M)\to H^i(G,M)$ is an isomorphism for all $i\geq 0$.
If, for all primes $p$ such that $\Z/p\in\cC$, the group $G$ is $p$-good in $\cC$ we say that $G$ is \emph{$\cC$-good}. 
If $\cC$ is the class of all finite groups, we just say that the group $G$ is \emph{good}.
\end{definition}

We will also need to impose some finiteness condition on cohomology:

\begin{definition}\label{fct}We say that a group $G$ is of \emph{finite virtual $p$-cohomological type} if:
\begin{enumerate}
\item $G$ has finite virtual $p$-cohomological dimension (briefly, $\vcd_p(G)<\infty$);
\item for every finite $\F_p[G]$-module $M$, the cohomology $H^i(G,M)$ is finite for $i\geq 0$.
\end{enumerate}
A group $G$ is of \emph{finite virtual cohomological type} if it is of finite virtual $p$-cohomological type for all primes $p>0$.
\end{definition}

Our first result is that $p$-goodness of a group $G$ of finite virtual $p$-cohomological type implies that $\iota_c$ is a bijection
when restricted to conjugacy classes of $p$-subgroups of $G$. 

Let us fix some terminology. For a group $L$, let us denote by $\cS_f(L)$, $\cS_{f\!s}(L)$ and $\cS_p(L)$, respectively, the set of finite,
finite solvable and finite $p$ (for a fixed prime $p$) subgroups of $L$. The group $L$ acts naturally on each of these sets by conjugation. 
We then denote by $\cS_f(L)/_{\sim}$, $\cS_{f\!s}(L)/_{\sim}$ and $\cS_p(L)/_{\sim}$, the respective quotients.

For a profinite group $R$, let us denote by $\cS(R)$ and $\cS(R)/_{\sim}$, respectively, the set of \emph{closed} subgroups of $R$ 
and the quotient of this set by conjugation. It is easy to see that both $\cS(R)$ and $\cS(R)/_{\sim}$ are profinite sets. We then endow the subsets 
$\cS_f(R)$, $\cS_{f\!s}(R)$, $\cS_p(R)$ of $\cS(R)$ and the subsets $\cS_f(R)/_{\sim}$, $\cS_{f\!s}(R)/_{\sim}$, $\cS_p(R)/_{\sim}$ of $\cS(R)/_{\sim}$ 
with the induced topologies. 

In general, the above subsets of finite subgroups are not closed in $\cS(R)$ and $\cS(R)/_{\sim}$, respectively, and so they are not profinite sets.
However, let us observe that the set $\cS_f(R)_{\leq k}$ (resp.\  $\cS_{f\!s}(R)_{\leq k}$, $\cS_p(R)_{\leq k}$) of finite 
(resp.\  finite solvable, finite $p$) subgroups of the profinite group $R$ of cardinality $\leq k$, for a fixed $k>0$, is a profinite set.  
Therefore, if the order of finite subgroups in $R$ is bounded by some constant (e.g.\ $R$ is virtually torsion free), the sets
$\cS_f(R)$, $\cS_{f\!s}(R)$ and $\cS_p(R)$ are profinite and then closed in $\cS(R)$. The same holds for 
the subsets $\cS_f(R)/_{\sim}$, $\cS_{f\!s}(R)/_{\sim}$ and $\cS_p(R)/_{\sim}$ of $\cS(R)/_{\sim}$.

A monomorphism $\iota\co G\hookra R$ of an abstract group $G$ in a profinite group $R$ induces an inclusion $\iota_f\co\cS_f(G)\hookra\cS_f(R)$. 
We endow $\cS_f(G)$ (resp.\ $\cS_{f\!s}(G)$ and $\cS_p(G)$) with the topology induced by this embedding and $\cS_f(G)/_{\sim}$ 
(resp.\ $\cS_{f\!s}(G)/_{\sim}$ and $\cS_p(G)/_{\sim}$) with the quotient topology. 
The embedding $\iota_f$ then induces a continuous map $\iota_c\co\cS_f(G)/_{\sim}\to\cS_f(R)/_{\sim}$ which is not necessarily an embedding.

Let us fix some more terminology.
For a subgroup $H$ of $G$, we denote by $N_G(H)$ and $C_G(H)$, respectively, the normalizer and centralizer of $H$ in $G$.
For a monomorphism $G\hookra R$ in a profinite group, we denote by $\ol{G}$ the closure of $G$ in $R$
for the profinite topology. We say that a property holds for $i\gg0$, if, for some positive integer $k$, the property holds for all $i>k$.
We then have (cf.\ Theorem~\ref{pcase} for the complete statement):

\begin{customtheorem}{A}For a fixed prime $p>0$, let $G$ be a group of finite virtual $p$-cohomological type. 
Let us assume that, for a monomorphism $\iota\co G\hookra R$ into a profinite group $R$ with dense image, 
the natural induced map $H^i(R,M)\to H^i(G, M)$ is an isomorphism for every discrete $\F_p[[R]]$-module $M$ and $i\gg0$. Then: 
\begin{enumerate}
\item $\iota$ induces a bijection of finite sets $\iota_c\co\cS_p(G)/_{\sim}\sr{\sim}{\to}\cS_p(R)/_{\sim}$. 
\item $p$-torsion elements of $G$ are $\cR$-conjugacy distinguished (cf.\ Definition~\ref{distconjsep}).
\item For every finite $p$-subgroup $H$ of $G$, we have the identities $N_{R}(H)=\ol{N_G(H)}$ and $C_{R}(H)=\ol{C_G(H)}$.
\end{enumerate}
\end{customtheorem}

\begin{remark}\label{Symonds}The above theorem should be compared with Symonds' classical result on cohomology isomorphisms of groups 
(cf.\ \cite[Theorem~1.1]{Symonds}). Note however that the latter only holds when both $H$ and $G$ belong to the same class of groups. 
Observe also that, in Theorem~A, we assume that $G$ has finite virtual $p$-cohomological type while no such finiteness hypotheses are needed 
in \cite[Theorem~1.1]{Symonds} for the case of profinite groups.
\end{remark}

As a special case of Theorem~A, we then get:

\begin{customcorollary}{B}Let $G$ be a $p$-good, residually $p$ group of finite virtual $p$-cohomological type 
and let $\iota\co G\hookra G_{\wC}$ be the natural monomorphism. Then, for every class of groups $\cC$ such that $\Z/p\in\cC$, we have:
\begin{enumerate}
\item $\iota$ induces a bijection of finite sets $\iota_c\co\cS_p(G)/_{\sim}\sr{\sim}{\to}\cS_p(G_{\wC})/_{\sim}$.  
\item $p$-torsion elements of $G$ are $\cC$-conjugacy distinguished (cf.\ Definition~\ref{distconjsep}).
\item For every finite $p$-subgroup $H$ of $G$, we have the identities $N_{G_\wC}(H)=\ol{N_G(H)}$ and $C_{G_\wC}(H)=\ol{C_G(H)}$.
\end{enumerate}
\end{customcorollary}

The proof of Theorem~A is based on a profinite version by Scheiderer (cf.\ Theorem~\ref{BrownRibes}) of a classical theorem of Brown
which relates the cohomology of a group of finite virtual $p$-cohomological dimension with the the cohomology of the normalizers of 
finite $p$-subgroups (cf.\ \cite[Corollary~7.4 in Chapter~X]{Brown}). 

Corollary~B applies to the profinite completion of virtually compact special groups and mapping class groups of surfaces of genus $\leq 2$
(more generally, to hyperelliptic mapping class groups, cf.\ Theorem~\ref{Hyptorsion} for a more precise result).

With more restrictive hypotheses, we extend Theorem~A to finite solvable subgroups.
More precisely, we have the following theorem (cf.\ Theorem~\ref{maintheorem} for the complete statement):

\begin{customtheorem}{C}Let $G$ be a group such that:
\begin{enumerate}
\item $G$ is $\cC$-good, residually $\cC$ and of finite virtual $p$-cohomological type for all primes $p$ such that $\Z/p\in\cC$.
\item For every solvable $\cC$-subgroup $H$ of $G$, the normalizer $N_G(H)$ of $H$ in $G$ satisfies the previous hypothesis (i)
and, moreover, its closure $\ol{N_G(H)}$ in the pro-$\cC$ completion $G_{\wC}$ of $G$ coincides with the pro-$\cC$
completion of $N_G(H)$.
\end{enumerate}
Let $\iota\co G\hookra G_{\wC}$ be the natural monomorphism. Then:
\begin{enumerate}
\item $\iota$ induces a bijection of finite sets $\iota_c\co\cS_{f\!s}(G)/_{\sim}\sr{\sim}{\to}\cS_{f\!s}(G_{\wC})/_{\sim}$. 
\item $\cC$-torsion elements are $\cC$-conjugacy distinguished (cf.\ Definition~\ref{distconjsep}).
\item For every $\cC$-subgroup $H$ of $G$, there holds $C_{G_\wC}(H)=\ol{C_G(H)}$. If, moreover, $H$ is solvable,
we also have $N_{G_\wC}(H)=\ol{N_G(H)}$.
\end{enumerate}
\end{customtheorem}

As a special case of Theorem~C, we then get:

\begin{customcorollary}{D}Let $G$ be a group such that:
\begin{enumerate}
\item $G$ is residually finite, good and of finite virtual cohomological type.
\item For every finite solvable subgroup $H$ of $G$, the normalizer $N_G(H)$ of $H$ in $G$ satisfies the previous hypothesis (i)
and, moreover, its closure $\ol{N_G(H)}$ in the profinite completion $\hG$ of $G$ coincides with the profinite completion of $N_G(H)$.
\end{enumerate}
Let $\iota\co G\hookra \hG$ be the natural monomorphism. Then:
\begin{enumerate}
\item $\iota$ induces a bijection of finite sets $\iota_c\co\cS_{f\!s}(G)/_{\sim}\sr{\sim}{\to}\cS_{f\!s}(\hG)/_{\sim}$.  
\item Torsion elements are conjugacy distinguished.
\item For every finite subgroup $H$ of $G$, there holds $C_{\hG}(H)=\ol{C_G(H)}$. If, moreover, $H$ is  solvable,
we also have $N_{\hG}(H)=\ol{N_G(H)}$.
\end{enumerate}
\end{customcorollary}

Corollary~D  applies to virtually compact special toral relatively hyperbolic groups (these include Fuchsian groups, 
fundamental groups of $3$-orbifolds and of uniform standard arithmetic hyperbolic orbifolds,  one-relator groups with torsion 
and some groups of small cancellation). 
\bigskip

\noindent
{\bf Acknowledgements.} The second author is grateful to Ashot Minasyan for helping with the references in Section 6.2.

\section{Hereditary conjugacy separability}
For the proof of Theorem~\ref{maintheorem}, we need to establish a relation between a stronger version of conjugacy
separability and centralizer and normalizer of finite subgroups in the profinite completions of a group. 

For a monomorphism $\iota\co G\hookra R$ of a group $G$ into a profinite group $R$, let us denote by $\cR$ the set of finite groups
which are obtained as images of $G$ in the finite discrete quotients of $R$. We say that a subgroup $U$ of $G$ is \emph{$\cR$-open} 
if, for some normal subgroup $N$ of $G$ contained in $U$, we have $G/N\in\cR$. We call the topology induced on $G$ by $\iota$ the 
\emph{pro-$\cR$ topology of $G$}. Note that, for $\cC$ a class of finite groups and $R=G_\wC$, we have that $\cR\subseteq\cC$ 
and the pro-$\cR$ topology of $G$ coincides with the pro-$\cC$ topology.

\begin{definition}\label{distconjsep}Let $\iota\co G\hookra R$ be as above.
\begin{enumerate}
\item An element $x$ of $G$ is $\cR$-conjugacy distinguished if its conjugacy class $x^G$ is closed for the pro-$\cR$ topology of $G$. 
The element $x$ is \emph{hereditarily $\cR$-conjugacy distinguished} if, for every $\cR$-open subgroup $U$ 
of $G$, the conjugacy class $x^U$ is closed for the pro-$\cR$ topology of $G$. 
\item A finite subgroup $H$ of $G$ is \emph{subgroup $\cR$-conjugacy distinguished} if its conjugacy class $H^G$ is closed in the space $\cS_f(G)$.
The subgroup $H$ is \emph{hereditarily subgroup $\cR$-conjugacy distinguished} if, for every $\cR$-open subgroup $U$ 
of $G$, the conjugacy class $H^U$ is closed in the space $\cS_f(G)$.
\end{enumerate}
\end{definition}

Note that, for $V$ an open subgroup of $R$, there is a natural topological embedding $\cS_f(V)\subseteq\cS_f(R)$ and then, 
for $U=\iota^{-1}(V)$, a natural topological embedding $\cS_f(U)\subseteq\cS_f(G)$.
We have the following two criteria which characterize the property of an element $x\in G$ (resp.\ a finite subgroup $H<G$) 
being hereditarily (resp.\ subgroup) conjugacy distinguished in terms of its centralizer (resp.\ normalizer) in the profinite completion $\hG$:

\begin{lemma}\label{centcriterion}For $G$ and $\cR$ as above. The following statements are equivalent:
\begin{enumerate}
\item The element $x\in G$ is hereditarily $\cR$-conjugacy distinguished.
\item There is a fundamental system of neighborhoods of the identity $\{U_\ld\}_{\ld\in\Ld}$ for the $\cR$-topology on $G$, consisting 
of open normal subgroups, such that $x$ is $\cR$-conjugacy distinguished in $\la x\ra U_\ld$, for all $\ld\in\Ld$.
\item $x$ is $\cR$-conjugacy distinguished in $G$ and $C_{R}(x)=\ol{C_G(x)}$.
\end{enumerate}
\end{lemma}

\begin{proof}(i)$\Ra$(ii): This follows from the definitions.
\smallskip

\noindent
(ii)$\Ra$(iii): If, for some $\cR$-open subgroup $U$ of $G$, the element $x$ is $\cR$-conjugacy distinguished in $\la x\ra U$,
then the conjugacy orbit $x^U$ is closed for the pro-$\cR$ topology of $G$, but this implies that the conjugacy orbit $x^G$ is closed as well
and so $x$ is $\cR$-conjugacy distinguished in $G$. Let us then show that $C_{R}(x)=\ol{C_G(x)}$.

By hypothesis, there is a fundamental system of neighborhood of the identity $\{U_\ld\}_{\ld\in\Ld}$ for the $\cR$-topology on $G$, consisting 
of open normal subgroups, such that $x^{U_\ld}$ closed for the pro-$\cR$ topology of $G$.
By Lemma~3.7 in \cite{Minasyan} (the lemma remains true if, instead of the full profinite topology, we take on $G$ any residual profinite topology), 
where, in the lemma, we take $H=G$ and $K=U_\ld$, there is a normal subgroup $L_\ld$, open for the $\cR$-topology and contained in $U_\ld$,
such that, if we denote by $p_\ld\co G\to G/L_\ld$ the natural epimorphism, there holds $C_{G/L_\ld}(p_\ld(x))\subseteq p_\ld(C_G(x)U_\ld)$, for
all $\ld\in\Ld$.

Note that $C_R(x)=\varprojlim_{\ld\in\Ld}C_{G/L_\ld}(p_\ld(x))$ and $\ol{C_G(x)}=\bigcap_{\ld\in\Ld}C_G(x)\ol{U}_\ld$, where $\ol{U}_\ld$ is the
closure of $U_\ld$ in the profinite group $R$. Therefore, the above system of inclusion implies that $C_{R}(x)\subseteq\ol{C_G(x)}$. 
Since the reverse inclusion is obvious, the conclusion follows.
\smallskip

\noindent
(iii)$\Ra$(i): It is easy to see that the argument in the proof of Proposition~3.1 in \cite{CZ} applies element by element.
\end{proof}

\begin{lemma}\label{normcriterion}For $G$ and $\cR$ as above. The following statements are equivalent:
\begin{enumerate}
\item The finite subgroup $H$ of $G$ is hereditarily subgroup $\cR$-conjugacy distinguished.
\item There is a fundamental system of neighborhoods of the identity $\{U_\ld\}_{\ld\in\Ld}$ for the $\cR$-topology on $G$, consisting 
of open normal subgroups, such that $H$ is subgroup $\cR$-conjugacy distinguished in $H U_\ld$, for all $\ld\in\Ld$.
\item $H$ is subgroup $\cR$-conjugacy distinguished in $G$ and $N_R(H)=\ol{N_G(H)}$.
\end{enumerate}
\end{lemma}

\begin{proof}The proof is essentially the same as that of Lemma~\ref{centcriterion}, with some obvious modifications.
\end{proof}

\section{A Brown theorem for profinite groups}\label{BrownSec}
In order to prove our results, we will need a version for profinite groups of Brown theorem \cite[Corollary~7.4 in Chapter~X]{Brown}. 
We restrict to profinite groups of the following type:

\begin{definition}\label{fctpro}A profinite group $R$ is of \emph{finite virtual $p$-cohomological type} if:
\begin{enumerate}
\item $R$ has finite virtual $p$-cohomological dimension (briefly, $\vcd_p(R)<\infty$);
\item for every finite discrete $\F_p[[R]]$-module $M$, the cohomology group $H^i(R,M)$ is finite for $i\geq 0$.
\end{enumerate}
\end{definition}

For a prime number $p>0$, let $\mathcal T:=\cS_p(R)_{\leq p}$ be the profinite set of subgroups of order $p$ of $R$ 
and let $\cT/_\sim$ be its quotient by the conjugacy action of $R$. This is also a profinite set and there is a natural continuous surjective map
$\mu\co\cT\to\cT/_\sim$. For $\tau\in\cT/_\sim$, the inverse image $\mu^{-1}(\tau)$ is also a profinite set.
Let then $\F_p[[\cT]]$ (resp.\ $\F_p[[\mu^{-1}(\tau)]]$) be the free profinite $\F_p$-module on the profinite space $\cT$ (resp.\ $\mu^{-1}(\tau)$). 

For a $\F_p[[R]]$-module $A$, let: 
\[C(\cT,M)=\hom(\F_p[[\cT]],M)\hspace{0.5cm} (\mbox{resp.}\hspace{0.5cm} C(\mu^{-1}(\tau),M):=\hom(\F_p[[\mu^{-1}(\tau)]],M)\;)\]
be the group of continuous maps from $\cT$ (resp.\ $\mu^{-1}(\tau)$) to $M$. 

Endowed with the compact open topology, these are both discrete topological spaces and then have a natural structure 
of (left) discrete $\F_p[[R]]$-modules. Note that there is a natural monomorphism of $\F_p[[R]]$-modules: 
\[\epsilon\co M\hookra C(\cT,M),\] 
which sends $a\in M$ to the constant map $\epsilon_a\co\cT\to M$ with value $a$. 

\begin{theorem}[Scheiderer]\label{BrownRibes}Let $R$ be a profinite group of finite virtual $p$-cohomological type and
without subgroups isomorphic to $C_p\times C_p$, where $C_p$ denotes the multiplicative group of order $p$, 
and let $M$ be a finite discrete $\F_p[[R]]$-module. 
\begin{enumerate}
\item The natural homomorphism $\epsilon\co M\hookra C(\cT,M)$ induces a natural homomorphism, for $i\geq 0$: 
\[\epsilon^i\co H^i(R,M)\to\prod_{\tau\in\cT/_\sim}H^i(R,C(\mu^{-1}(\tau),M)),\]
 which is an isomorphism for $i>\vcd_p(R)$.
\item For all $T\in\mu^{-1}(\tau)$, there is an isomorphism:
\[H^i(R,C(\mu^{-1}(\tau),M))\cong H^i(N_R(T),M).\]
\end{enumerate}
\end{theorem}

\begin{proof}(i): This is a more intrinsic formulation (we do not choose representatives for the orbits $\cT/_\sim$) of Corollary~12.19 in \cite{Scheiderer}.
Note that, since we are assuming that $H^i(R,M)$ is finite for $i\geq 0$, we conclude that $\epsilon^i$ is an isomorphism, for $i>\vcd_p(R)$, and not just
that it has dense image as in Scheiderer's original formulation.
\smallskip

\noindent
(ii): The isomorphism follows from Shapiro's lemma and the remark that $\mu^{-1}(\tau)$, as a profinite $G$-set, is isomorphic to the coset space 
$R/N_R(T)$, so that $C(\mu^{-1}(\tau),M)$ is isomorphic to the coinduced module $\operatorname{Coind}^R_{N_R(T)}(M)$.
\end{proof}

\section{The case of finite $p$-subgroups}
We will actually prove a stronger version of Theorem~A. For this, we need the following refinement of Definition~\ref{fct}:

\begin{definition}\label{fct2}Given a monomorphism $\iota\co G\hookra R$ of a group $G$ into a profinite group $R$
We say that the group $G$ is of \emph{finite virtual $p$-cohomological type with respect to $\iota$} if:
\begin{enumerate}
\item $G$ has finite virtual $p$-cohomological dimension (briefly, $\vcd_p(G)<\infty$);
\item for every finite discrete $p$-primary $\F_p[[R]]$-module $M$, the cohomology $H^i(G,M)$ is finite for $i\geq 0$.
\end{enumerate}
\end{definition}

Note that the profinite topology on $R$ naturally defines a profinite topology not only on every subgroup $H$ of $G$ but also on every subquotient $H/L$
of $G$ by a finite subgroup $L$. We call this topology the \emph{pro-$\cR$ topology of $H/L$}. For the class of groups consisting of such subquotients, 
we say that two groups are $\cR$-commensurable if they contain isomorphic $\cR$-open subgroups.

\begin{lemma}\label{equivfct}The second condition of Definition~\ref{fct2} is equivalent to the condition that, for a fundamental 
system $\{U_\ld\}_{\ld\in\Ld}$ of neighborhoods of the identity for the pro-$\cR$ topology of $G$, consisting of open normal 
subgroups, the cohomology $H^i(U_\ld,\F_p)$ is finite for $i\geq 0$. In particular, for the class of subquotients of $G$ by finite subgroups, 
being of finite virtual $p$-cohomological type with respect to the embedding in their pro-$\cR$ completion is a property of 
the $\cR$-commensurability class of the group.
\end{lemma}

\begin{proof}Let us assume that $G$ satisfies the second condition of Definition~\ref{fct2}. Then, by Shapiro's lemma, 
$H^i(U,\F_p)=H^i(G,\operatorname{Coind}^G_U(\F_p))$ is finite for every $\cR$-open subgroup $U$ of $G$ and all $i\geq 0$. 

Let us then assume that $G$ satisfies the condition in the lemma and let $M$ be a finite discrete $\F_p[[R]]$-module.
For some $\ld\in\Ld$, the restriction of $M$ to $U_\ld$ is then the trivial module and
the condition of the lemma implies that $H^i(U_\ld,M)$ is finite for all $i\geq 0$. The Lyndon-Hochschild-Serre spectral sequence
associated to the short exact sequence $1\to U_\ld\to G\to G/U_\ld\to 1$ is the first quadrant spectral sequence:
\[E_2^{p,q}=H^p(G/U_\ld,H^q(U_\ld,M))\Ra H^{p+q}(G,M).\]
Since the terms $E_2^{p,q}=H^p(G/U_\ld,H^q(U_\ld,M))$ are finite for all $p,q\geq 0$, the abutment $E_\infty^{p,q}$ of the spectral sequence is also finite. 
The cohomology group $H^i(G,M)$ then admits a (finite) filtration by finite vector spaces and is thus finite.
\end{proof}

We then have:

\begin{theorem}\label{pcase}Let  $\iota\co G\hookra R$ be a monomorphism of a group $G$ into a profinite group $R$ with dense image. 
For a fixed prime $p>0$, let us assume that $G$ is of finite virtual $p$-cohomological type with respect to $\iota$ and that
the natural map $H^i(R,M)\to H^i(G, M)$ is an isomorphism for every discrete $\F_p[[R]]$-module $M$ and $i\gg 0$. Then, we have: 
\begin{enumerate}
\item $\iota$ induces a bijection of finite sets $\iota_c\co\cS_p(G)/_{\sim}\sr{\sim}{\to}\cS_p(R)/_{\sim}$. 
\item $p$-torsion elements and finite $p$-subgroups of $G$ are hereditarily (resp.\ subgroup) $\cR$-conjugacy distinguished.
\item For every finite $p$-subgroup $H$ of $G$, there holds: 
\[N_{R}(H)=\ol{N_G(H)}\hspace{0.5cm} \mbox{and}\hspace{0.5cm} C_{R}(H)=\ol{C_G(H)}.\]
\item For every discrete $\F_p[[N_{R}(H)]]$-module $M$ and all $i\geq 0$, 
the natural induced map $H^i(N_{R}(H),M)\to H^i(N_G(H), M)$ is an isomorphism. 
\item The quotient $N_G(H)/H$ is of finite virtual $p$-cohomological type with respect to the monomorphism $N_G(H)/H\hookra N_{R}(H)/H$
induced by $\iota$.
\end{enumerate}
\end{theorem}

\subsection{A preliminary lemma}
The proof of Theorem~\ref{pcase} is by induction on the order of finite subgroups of $G$ and $R$. The base for the induction
is provided by the case of subgroups of order $p$. We will need the following lemma:

\begin{lemma}\label{prop:prime_order}With the same hypotheses of Theorem~\ref{pcase}, let us assume moreover that $R=V\rtimes C_p$ is a 
semidirect product, where $C_p$ is a group of prime order $p$ and $V$ is a profinite group such that $\cd_p(V)<\infty$. Then:
\begin{enumerate} 
\item $\iota$ induces a bijection of finite sets $\iota_c\co\cS_p(G)/_{\sim}\sr{\sim}{\to}\cS_p(R)/_{\sim}$. 
\item Let $U:=\iota^{-1}(V)$. Then, for every subgroup $T$ of $G$ of order $p$, there holds:
\[N_G(T)=T\times(U\cap N_G(T))=C_G(T)\hspace{0.5cm}\mbox{and}\hspace{0.5cm} N_R(T)=T\times(V\cap N_R(T))=C_R(T).\]
\item For every subgroup $T$ of $G$ of order $p$, the induced map on cohomology groups $H^i (V\cap N_{R}(\iota(T)),\F_p)\to H^i(U\cap N_G(T),\F_p)$ 
is an isomorphism for all $i\geq 0$, where $\F_p$ is the trivial module. 
\item For every subgroup $T$ of $G$ of order $p$, there holds $\vcd_p(N_{G}(T))<\infty$ and the cohomology $H^i(U\cap N_G(T),\F_p)$
is a finite vector space for all $i\geq 0$.
\end{enumerate}
\end{lemma}

\begin{proof}Since $\cd_p(V)<\infty$, the group $V$ is $p$-torsion free. Hence all finite $p$-subgroups of $R=V\rtimes C_p$
are isomorphic to $C_p$ and our hypotheses imply that $R$ is of finite virtual $p$-cohomological type. By Theorem~\ref{BrownRibes}, 
for every finite discrete $\F_p[[R]]$-module $M$ and $i>\vcd_p(R)$, there is then a natural isomorphism (with the notations of Section~\ref{BrownSec}):
\begin{equation}\label{eq:Ribes}
H^i(R, M)\sr{\sim}{\to}\prod_{\tau\in\cT/_\sim}H^i(R,C(\mu^{-1}(\tau),M)).
\end{equation}
Moreover, for all $\tau\in\cT/_\sim$ and $T\in\mu^{-1}(\tau)$, there is an isomorphism: 
\begin{equation}\label{eq:Ribes2}
H^i(R,C(\mu^{-1}(\tau),M))\cong H^i(N_R(T),M).
\end{equation}

Since, for all $T\in\cT$, the subgroup $T$ of $R$ has trivial intersection with the torsion free normal subgroup $V$, 
there is a direct product decomposition: 
\[N_R(T)=T\times(V\cap N_R(T)),\]
which proves item (ii) of the lemma. 

By the K\"unneth formula for group cohomology and the fact that $H^k(T,\F_p)=\F_p$ for all $k\geq 0$,
there is then a series of natural isomorphisms, for all $n\geq 0$:
\begin{equation}\label{Kunneth1}
H^n(N_R(T),\F_p)\cong\bigoplus_{i+k=n} H^i(V\cap N_R(T),\F_p)\otimes H^k(T,\F_p)\cong\bigoplus_{i\leq n} H^i(V\cap N_R(T),\F_p).
\end{equation}
In particular, the cohomology group $H^n(N_R(T),\F_p)$ contains $H^0(V\cap N_R(T),\F_p)\cong\F_p$ as a direct summand and is nontrivial.
Since, by hypothesis, $H^i(R, \F_p)$ is finite for $i\geq 0$, the isomorphism~\eqref{eq:Ribes} then implies that the set $\cT/_\sim$ is finite.

Note that, since the group $G$ identifies with a subgroup of $R$, all elementary abelian $p$-subgroups of $G$ have also rank at most $1$.
Let $\cI$ be the set of subgroups of order $p$ in $G$, let $\cI/_\sim$ be its quotient by the conjugacy action of $G$ and $m\co\cI\to\cI/_\sim$ the natural map.  
For $\alpha\in \cI/_\sim$ and a given $p$-primary torsion $G$-module $M$, let also $C(m^{-1}(\alpha),M):=\hom(\F_p[m^{-1}(\alpha)],M)$ with its natural
structure of left $G$-module.

By hypothesis, $\vcd_p(G)<\infty$. So we can apply a classical result of Brown (cf.\ Corollary~7.4 and Remark below in Chapter~X of \cite{Brown}) 
and conclude that, for every $p$-primary torsion $G$-module $M$ and all $i>\vcd_p(G)$, there is a natural isomorphism:
\begin{equation}\label{eq:Brown} 
H^i(G,M)\sr{\sim}{\to} \prod_{\alpha\in \cI/_\sim} H^i(G,C(m^{-1}(\alpha),M)).
\end{equation}
Moreover, for all $\alpha\in \cI/_\sim$ and $T\in m^{-1}(\alpha)$, there is an isomorphism: 
\begin{equation}\label{eq:Brown2}
H^i(G,C(m^{-1}(\alpha),M))\cong H^i(N_G(T),M).
\end{equation}

 For all $T\in \cI$, there is a direct product decomposition: 
\[N_G(T)=T\times(U\cap N_G(T)).\]
As above, by the K\"unneth formula, we deduce that there is a series of natural isomorphisms, for all $n\geq 0$:
\begin{equation}\label{Kunneth2}
H^n(N_G(T),\F_p)\cong\bigoplus_{i+k=n} H^i(U\cap N_G(T),\F_p)\otimes H^k(T,\F_p)\cong\bigoplus_{i\leq n} H^i(U\cap N_G(T),\F_p).
\end{equation}
In particular, the cohomology group $H^n(N_G(T),\F_p)$ contains $H^0(U\cap N_G(T),\F_p)\cong\F_p$ as a direct summand and is nontrivial.
Since, by hypothesis, $H^i(G, \F_p)$ is finite for $i\geq 0$, this implies that the set $\cI/_\sim$ is finite. It also follows that $H^i(U\cap N_G(T),\F_p)$
is finite for all $i\geq 0$, which is the claim of the second part of item (iv) of the lemma while
the first part of the same item simply follows from Shapiro's lemma. 

\smallskip
Putting together the isomorphisms~\eqref{eq:Ribes} and~\eqref{eq:Brown}, for every finite discrete $\F_p[[R]]$-module $M$ and all 
$i>\vcd_p(G)$, we get a natural commutative diagram:
\begin{equation}\label{commdiagr}
\xymatrix{H^i(G,M)\ar[r]& \prod_{\alpha \in \cI/_\sim} H^i(G,C(m^{-1}(\alpha),M)) \\	
H^i(R,M)\ar[u]^{\iota^\ast}	\ar[r]&\prod_{\tau\in\cT/_\sim} H^i(R,C(\mu^{-1}(\tau),M))\ar[u]^{\iota_c^\ast}	}, 
\end{equation}
induced by the embedding $\iota\co G\hookra R$ and the natural map $\iota_c\co \cI/_\sim\to\cT/_\sim$.
The left vertical map and the horizontal maps are isomorphisms. Hence, the right vertical map is also an isomorphism. 

For $M=\F_p$ the trivial module, $\alpha\in\cI/_\sim$ and $T\in m^{-1}(\alpha)$, the isomorphism $\iota_c^\ast$ in the diagram maps
$H^i(R,C(\mu^{-1}(\iota_c(\alpha)),\F_p))$ to $H^i(G,C(m^{-1}(\alpha),\F_p))$. This implies that $\iota_c^\ast$ can be an isomorphism 
only if the natural map $\iota_c\co \cI/_\sim\to\cT/_\sim$ is in fact a bijection. This proves the first item of the lemma. 

The horizontal isomorphisms in diagram~\eqref{commdiagr} implies that, for every subgroup $T$ of $G$, 
the natural map $H^i(N_R(\iota(T)), \F_p) \to H^i(N_G(T), \F_p)$ is an isomorphism of finite vector spaces for all $i>\vcd_p(G)$. 
By the K\"unneth decompositions~\eqref{Kunneth1} and~\eqref{Kunneth2}, we then have that the natural map
$H^i (V\cap N_{R}(\iota(T)),\F_p)\to H^i(U\cap N_G(T),\F_p)$ is an isomorphism of finite vector spaces for all $i\geq 0$ 
(as the homomorphism sends each direct summand into the corresponding direct summand). This proves item (iii) of the lemma.
\end{proof}

\subsection{Theorem~\ref{pcase} holds for subgroups of order $p$ of $G$ and $R$}\label{orderp}
(i): The hypotheses of Theorem~\ref{pcase} imply that $\vcd_p(R)<\infty$. Hence, $R$ contains an open normal subgroup $V$ such that 
$\cd_p(V)<\infty$ which implies that $V$ is $p$-torsion free. The inverse image $U=\iota^{-1}(V)$ is then an $\cR$-open normal
$p$-torsion free subgroup of $G$. 

Let $T$ be a subgroup of order $p$ of $R$, put $K:=V T$ and let $L=\iota^{-1}(K)$.
The group $L$ together with the restriction $\iota|_L\co L\hookra K$ clearly satisfy all the hypotheses of Theorem~\ref{pcase}.
Therefore, by the first item of Lemma~\ref{prop:prime_order}, every finite subgroup of $K$ of order $p$ is conjugate to a subgroup of $L$.
In particular, $T$ is conjugate to a subgroup of $G$. We have then proved that the natural map $\iota_c\co\cS_p(G)/_\sim\to\cS_p(R)/_\sim$ is surjective.

Let us show that the subgroups of order $p$ of $G$ are hereditarily subgroup $\cR$-conjugacy distinguished. 
Besides the second half of item (ii) of the theorem, this implies that $\iota_c$ is injective. 

By Lemma~\ref{normcriterion}, in order to prove that a subgroup $T$ of order $p$ of $G$ is hereditarily subgroup $\cR$-conjugacy distinguished,
it is enough to show that, for a fundamental system $\{V_\ld\}_{\ld\in\Ld}$ of neighborhoods of the identity of $R$ consisting of open normal 
subgroups contained in $V$ and $U_\ld:=\iota^{-1}(V_\ld)$, for $\ld\in\Ld$, the embedding $U_\ld C\hookra V_\ld C$ induces a 
bijection between conjugacy classes of subgroups of order $p$. This follows from Lemma~\ref{prop:prime_order}. In fact, the hypotheses
of Theorem~\ref{pcase} imply that every embedding $U_\ld C\hookra V_\ld C$ satisfies the hypotheses of Lemma~\ref{prop:prime_order},
for $\ld\in\Ld$.

The fact that the set of conjugacy classes of finite subgroups of order $p$ of $G$ is finite follows from the fact that, by (i) of Lemma~\ref{prop:prime_order}, 
there are only finitely many such conjugacy classes in $UT$ for every subgroup $T$ of $G$ of order $p$ and the fact that, if two such 
$T_1$ and $T_2$ have the same image in the finite quotient $G/U$, then $UT_1=UT_2$.
\smallskip

\noindent
(ii): The second half of item (ii) has been proved above. Let us show that elements of order $p$ in $G$ are hereditarily $\cR$-conjugacy distinguished.
With the above notations, item (i) and (ii) of Lemma~\ref{prop:prime_order} imply that an element $x\in G$ of order $p$ is conjugacy
distinguished in $U_\ld\la x\ra$, for all $\ld\in\Ld$. The conclusion then follows from Lemma~\ref{centcriterion}.
\smallskip

\noindent
(iii): This follows from the previous two items, Lemma~\ref{normcriterion} and Lemma~\ref{centcriterion}.
\smallskip

\noindent
(iv): By item (iii) of Lemma~\ref{prop:prime_order}, we have that, for every subgroup $T$ of $G$ of order $p$, the natural map
$H^i (V_\ld\cap N_{R}(\iota(T)),\F_p)\to H^i(U_\ld\cap N_G(T),\F_p)$ is an isomorphism for all $i\geq 0$ and all $\ld\in\Ld$,
where $\F_p$ is the trivial module. Since $\{V_\ld\}_{\ld\in\Ld}$ is a fundamental system of neighborhoods of the identity of $R$,
this implies  item (iv) of the theorem. 
\smallskip

\noindent
(v): This immediately follows from item (iv) of Lemma~\ref{prop:prime_order} and Lemma~\ref{equivfct}.

\subsection{Proof of Theorem~\ref{pcase}}\label{pcaseproof}
The proof is by induction on the cardinality of finite $p$-subgroups. The base for the induction is provided by Section~\ref{orderp}.
Let us then assume that Theorem~\ref{pcase} holds for finite $p$-subgroups and elements of order less than a fixed integer $k>0$ 
and let us prove it for all $p$-subgroups and elements of order $k$. 
\smallskip

\noindent
(i) and (ii):
Let $\iota\co G\hookra R$ be a monomorphism as in the hypotheses of the theorem and let $H$ be a subgroup of $R$ of order $k$. 
Since $\vcd_p(R)<\infty$, there is a fundamental system $\{V_\ld\}_{\ld\in\Ld}$ of neighborhoods of the identity of $R$ consisting of 
$p$-torsion free open normal subgroups. Let $R_\ld:=V_\ld H<R$, for $\ld\in\Ld$, and let us observe that these groups also satisfy 
the hypotheses of Theorem~\ref{pcase}. Note then that $H$ contains a proper nontrivial central subgroup $T$ of order $p$ so that 
$H<C_{R_\ld}(T)\leq N_{R_\ld}(T)$.

Let $G_\ld:=\iota^{-1}(R_\ld)$, for $\ld\in\Ld$. By the case treated in Section~\ref{orderp}, possibly after conjugating by an element of $R_\ld$, 
we can assume that $T$ is contained in $G_\ld$, that $N_{R_\ld}(T)=\ol{N_{G_\ld}(T)}$, that the natural map 
$H^i(N_{R_\ld}(T),M)\to H^i(N_{G_\ld}(T),M)$ is an isomorphism for every discrete $\F_p[[N_{R_\ld}(T)]]$-module $M$ and all $i\geq 0$ 
and that $N_{G_\ld}(T)/T$ is of finite virtual $p$-cohomological type with respect to the natural monomorphism $N_{G_\ld}(T)/T\hookra N_{R_\ld}(T)/T$.

\begin{lemma}\label{inductionok}The induced map on cohomology $H^i(N_{R_\ld}(T)/T,M)\to H^i(N_{G_\ld}(T)/T,M)$ 
is an isomorphism for every discrete $\F_p[[N_{R_\ld}(T)/T]]$-module $M$ and all $i\geq 0$.
\end{lemma}

\begin{proof}The groups $N_{R_\ld}(T)=\ol{N_{G_\ld}(T)}$ and $N_{R_\ld}(T)/T=\ol{N_{G_\ld}(T)}/T$ are commensurable. Let then
$\bar K$ be a normal open subgroup of $N_{R_\ld}(T)/T$ which identifies with a normal open subgroup of $N_{R_\ld}(T)$ and let
$K$ be the inverse image of this group in $N_{G_\ld}(T)/T$. Then, $K$ also identifies with a normal finite index subgroup of $N_{G_\ld}(T)$.
From Shapiro's lemma and the above assumptions, it follows that, for every $\F_p[[\bar K]]$-module $M$ and all $i\geq 0$,
the induced map on cohomology $H^i(\bar K,M)\to H^i(K,M)$ is an isomorphism. 

By the Lyndon-Hochschild-Serre spectral sequence,
applied to the short exact sequences $1\to\bar K\to N_{R_\ld}(T)/T\to\bar C\to 1$ and $1\to K\to N_{G_\ld}(T)/T\to C\to 1$, 
where $\bar C\equiv C$ are the respective cokernels, the homomorphism $N_{G_\ld}(T)/T\hookra N_{R_\ld}(T)/T$ 
then induces the isomorphism claimed in the statement of the lemma.
\end{proof}

By the above remarks and Lemma~\ref{inductionok}, the homomorphism $\iota_T\co N_{G_\ld}(T)/T\hookra N_{R_\ld}(T)/T$, induced by $\iota$,
satisfies the hypotheses of Theorem~\ref{pcase}. From the induction hypothesis, it then follows that Theorem \ref{pcase} holds
for subgroups and elements of order $k/p$ of $N_{G_\ld}(T)/T$ and $N_{R_\ld}(T)/T$, for all $\ld\in\Ld$.

From item (i) of Theorem \ref{pcase}, applied to the finite $p$-subgroup $H/T$ of $N_{R_\ld}(T)/T$, 
we deduce that $H/T$ is conjugate in $N_{R_\ld}(T)/T$ to a finite $p$-subgroup of 
$N_{G_\ld}(T)/T$ and so  $H$ is conjugate in $N_{R_\ld}(T)$ to a finite $p$-subgroup of $N_{G_\ld}(T)\leq G$. This proves that the map
$\iota_c\co\cS_p(G)_{\leq k}/_\sim\to\cS_p(R)_{\leq k}/_\sim$ is surjective and concludes the inductive step
for the surjectivity part of item (i) of Theorem \ref{pcase}. 

Let us then prove that a subgroup $H$ of $G$ of order $k$ is hereditarily subgroup $\cR$-conjugacy distinguished. This will prove the subgroups 
part of item (ii) of Theorem \ref{pcase} and also imply that $\iota_c$ is injective when restricted to the set of subgroups of order $\leq k$,
thus completing the proof that $\iota_c\co\cS_p(G)_{\leq k}/_\sim\to\cS_p(R)_{\leq k}/_\sim$ is bijective. 

With the above notations, by Lemma~\ref{normcriterion}, it is enough to show that $H$ is subgroup $\cR$-conjugacy distinguished in $G_\ld$, 
for all $\ld\in\Ld$.
Let then $K$ be a subgroup of order $k$ of $G_\ld$ which is conjugate to $H$ by an element $r_\ld\in R_\ld$, for $\ld\in\Ld$.
By the induction hypothesis, after possibly conjugating by an element of $G_\ld$,  we can assume that $H$ and $K$ intersect
in a nontrivial central subgroup $A$ of order less than $k$ and also that $r_\ld$ preserves such subgroup, 
that is to say that $r_\ld\in N_{R_\ld}(A)$.

From the induction hypothesis, it follows that $N_{R_\ld}(A)=\ol{N_{G_\ld}(A)}$ and that the monomorphism $N_{G_\ld}(A)\hookra N_{R_\ld}(A)$
satisfies the hypotheses of Theorem \ref{pcase}. By the induction hypothesis, applied to the images 
$H/A$, $K/A$ and $\bar r_\ld$ of $H$, $K$ and $r_\ld$ in the quotient $N_{R_\ld}(A)/A$, we have that $H/A$ and $K/A$ are conjugate
by an element $\bar r_\ld'\in N_{G_\ld}(A)/A$ and then that $H$ and $K$ are conjugate by an element $r_\ld'\in N_{G_\ld}(A)\leq G_\ld$, 
for all $\ld\in\Ld$. Hence, $H$ is hereditarily subgroup $\cR$-conjugacy distinguished in $G$.

Let us now show that elements of order $k$ of $G$ are hereditarily $\cR$-conjugacy distinguished. Let $x\in G$ be such an element.
With the above notations, let also $U_\ld:=\iota^{-1}(V_\ld)$, for $\ld\in\Ld$. By the previous part of the proof, we know, in particular, 
that the subgroup $\la x\ra$ is subgroup $\cR$-conjugacy distinguished in $U_\ld\la x\ra$, for all $\ld\in\Ld$. From the fact that both groups 
$U_\ld\la x\ra\cong U_\ld\rtimes\la x\ra$ and $V_\ld\la x\ra$ retract onto $\la x\ra$, it easily follows that $x$ is $\cR$-conjugacy 
distinguished in $U_\ld\la x\ra$, for all $\ld\in\Ld$. We then conclude by Lemma~\ref{centcriterion}.

To prove that the set $\cS_p(G)_{\leq k}/_{\sim}$ is finite, let us just observe that: by (i) of Section~\ref{orderp}, the group $G$ contains only finitely many 
subgroups of order $p$; any finite $p$-subgroup $H$ of $G$ contains a central subgroup $C$ of order $p$; by the induction hypothesis, $N_G(C)/C$ 
contains, up to conjugation, only finitely many finite $p$-subgroups of order $\leq k/p$, so that $N_G(C)$ contains, up to conjugation, 
only finitely many finite $p$-subgroups of order $\leq k$. We thus conclude that $\cS_p(G)_{\leq k}/_{\sim}$ is finite.
\smallskip

\noindent
(iii): By Lemma~\ref{normcriterion} and Lemma~\ref{centcriterion}, we immediately conclude that $N_{R}(H)=\ol{N_G(H)}$ and $C_{R}(H)=\ol{C_G(H)}$, 
for all subgroups $H$ of $G$ of order $k$, which completes the inductive step for item (iii) of Theorem~\ref{pcase}. 
\smallskip

\noindent
(iv): Let  $H$ be a subgroup of $G$ of order $k$ and let $T$ be a central subgroup of $H$ of order $p$, as above.
Let us observe that $N_{R}(H)$ contains $N_{N_{R}(T)}(H)=N_{N_{R}(H)}(T)$ as a subgroup of finite index.
Similarly, $N_{N_{G}(T)}(H)$ has finite index in $N_{G}(H)$. 
Let us then observe that $N_{N_{R}(T)}(H)$ is also commensurable with $N_{N_{R}(T)/T}(H/T)$ and, similarly,
$N_{N_{G}(T)}(H)$ is commensurable with $N_{N_{G}(T)/T}(H/T)$.

By the induction hypothesis, applied to the subgroup $H/T$, and item (iv) of Theorem~\ref{pcase}, the natural homomorphism 
$N_{N_{G}(T)/T}(H/T)\hookra N_{N_{R}(T)/T}(H/T)$ has dense image and induces an isomorphism on cohomology. 

Since $N_{R}(H)$ and $N_{G}(H)$ are commensurable with $N_{N_{R}(T)/T}(H/T)$ and $N_{N_{G}(T)/T}(H/T)$, respectively, it follows that
the natural homomorphism $N_{G}(H)\hookra N_{R}(H)$ has also dense image. Moreover, by Shapiro's lemma and Lyndon-Hochschild-Serre 
spectral sequence (cf.\ the proof of Lemma~\ref{inductionok}), this homomorphism induces an isomorphism on cohomology.
Hence, (iv) of Theorem \ref{pcase} holds for the subgroup $H$ of $G$.
\smallskip

\noindent
(v):  By the induction hypothesis, applied to the subgroup $H/T$ of $N_{N_{G}(T)/T}(H/T)$, item (v) of Theorem~\ref{pcase} holds for 
the quotient of this group by $H/T$. Since the groups $N_{G}(H)/H$ and $N_{N_{G}(T)/T}(H/T)/(H/T)$ are $\cR$-commensurable,
by Lemma~\ref{equivfct}, item (v) of Theorem~\ref{pcase} then also holds for the quotient $N_{G}(H)/H$.

\section{The case of finite solvable subgroups}
We will actually prove the following stronger version of Theorem~C:

\begin{theorem}\label{maintheorem} Let $G$ be a group and $\cC$ a class of finite groups. Let us assume that:
\begin{enumerate}
\item $G$ is $\cC$-good, residually $\cC$ and of finite virtual $p$-cohomological type for all primes $p$ such that $\Z/p\in\cC$.
\item For every solvable $\cC$-subgroup $H$ of $G$, the normalizer $N_G(H)$ of $H$ in $G$ satisfies the previous hypothesis (i)
and, moreover, its closure $\ol{N_G(H)}$ in 
the pro-$\cC$ completion $G_{\wC}$ of $G$  coincides with the pro-$\cC$ completion of $N_G(H)$.
\end{enumerate}
Let $\iota\co G\hookra G_{\wC}$ be the natural monomorphism. Then, we have:
\begin{enumerate}
\item $\iota$ induces a bijection of finite sets $\iota_c\co\cS_{f\!s}(G)/_{\sim}\sr{\sim}{\to}\cS_{f\!s}(G_{\wC})/_{\sim}$. 
\item $\cC$-torsion elements and solvable $\cC$-subgroups of $G$ are hereditarily (resp.\ subgroup) $\cC$-conjugacy distinguished.
\item For any $\cC$-subgroup $H$ of $G$, there holds $C_{G_\wC}(H)=\ol{C_G(H)}$. If, moreover, $H$ is solvable,
we also have $N_{G_\wC}(H)=\ol{N_G(H)}$.
\end{enumerate}
\end{theorem}

For the proof of Theorem~\ref{maintheorem}, we will need the following lemma:

\begin{lemma}\label{indtrick}If $G$ is a group which satisfies the hypotheses of Theorem~\ref{maintheorem} and $A$ a finite abelian subgroup 
of $G$, then also the quotient group $N_{G}(A)/A$ satisfies the hypotheses of Theorem~\ref{maintheorem}.
\end{lemma}

\begin{proof}Since the groups $N_{G}(A)$ and $N_{G}(A)/A$ are $\cC$-commensurable,
hypothesis (ii) of Theorem~\ref{maintheorem} implies that $N_{G}(A)/A$ satisfies hypothesis (i) of Theorem~\ref{maintheorem}.
Let us then prove that, for every solvable $\cC$-subgroup $B$ of $N_{G}(A)/A$, the normalizer $N_{N_{G}(A)/A}(B)$ of $B$ in $N_{G}(A)/A$ 
also satisfies hypothesis (i) of Theorem~\ref{maintheorem}.

Let $\td{B}$ be the inverse image of $B$ in $N_{G}(A)$. Then, $\td{B}$ is also a solvable $\cC$-group which contains $A$ as a normal subgroup.
The normalizer $N_{N_{G}(\td{B})}(A)$ is a $\cC$-open subgroup of the normalizer $N_{G}(\td{B})$. Since by hypothesis (ii) of 
Theorem~\ref{maintheorem}, the latter group satisfies hypothesis (i) of Theorem~\ref{maintheorem}, $N_{N_{G}(\td{B})}(A)$
also satisfies hypothesis (i) of Theorem~\ref{maintheorem}.

Now, the image of $N_{N_{G}(\td{B})}(A)$ in $N_{G}(A)/A$ is the normalizer $N_{N_{G}(A)/A}(B)$.
Therefore, the two groups $N_{N_{G}(\td{B})}(A)$ and $N_{N_{G}(A)/A}(B)$ are $\cC$-commensurable and then, by Shapiro's lemma, 
the latter contains a $\cC$-open normal subgroup $K$ which satisfies hypothesis (i) of Theorem~\ref{maintheorem}. 
By Lyndon-Hochschild-Serre spectral sequence arguments similar to the ones in the proof of Lemma~\ref{inductionok} and Lemma~\ref{equivfct}, 
we then conclude that the group $N_{N_{G}(A)/A}(B)$ also satisfies all hypotheses of item (i) of Theorem~\ref{maintheorem}.
\end{proof}

\begin{proof}[Proof of Theorem~\ref{maintheorem}]
The proof is by induction on the cardinality of finite solvable subgroups of $G$. The base for the induction is provided by Theorem~\ref{pcase} 
which implies, under hypothesis (i) of Theorem~\ref{maintheorem}, all items of Theorem~\ref{maintheorem} for the finite $p$-subgroups 
of $G$ and its torsion elements of order a power of $p$, for $\Z/p\in\cC$. 

Let us then assume that Theorem~\ref{maintheorem} holds for solvable subgroups of $G$ of order less than a fixed integer $k>0$ and let us prove 
that its statement then holds for the solvable subgroups of $G$ of order $k$. 

Let $H$ be a solvable subgroup of $G_\wC$ of order $k$ and let $P(H)$ be the (finite) set of primes which divide the order of $H$.
By hypothesis, we have that, in particular, $\vcd_p(G)<\infty$ and $G$ is $p$-good, for $p\in P(H)$. This implies that $G_\wC$ contains a 
$P(H)$-torsion free open subgroup $V$. Let then $\{V_\ld\}_{\ld\in\Ld}$ be a fundamental system of neighborhoods of the identity of 
$G_\wC$ consisting of open normal subgroups contained in $V$, let $U_\ld=\iota^{-1}(V_\ld)$ and put $G_\wC^\ld:=V_\ld H$ and 
$G^\ld:=\iota^{-1}(G_\wC^\ld)$, for $\ld\in\Ld$. It is clear that $(G^\ld)_\wC\cong G_\wC^\ld$ and that $G^\ld$ satisfies the hypotheses 
of Theorem~\ref{maintheorem}, for all $\ld\in\Ld$.
\smallskip

\noindent
(i): The first step of the proof consists in showing that there is an element $x\in G_\wC^\ld$ such that $H^x$ is contained in $G^\ld\subset G_\wC^\ld$,
so that the map $\iota_c\co\cS_{f\!s}(G)/_{\sim}\to\cS_{f\!s}(G_{\wC})/_{\sim}$ is surjective when restricted to conjugacy classes of solvable
$\cC$-subgroups of order $k$.

The subgroup $H$ of $G_\wC^\ld$ contains a \emph{proper} nontrivial normal abelian subgroup $A$, so that
$H<N_{G_\wC^\ld}(A)$. By the induction hypothesis, possibly after conjugating by an element of $G_\wC^\ld$, 
we can assume that $A$ is contained in $G^\ld$ and then, again by the induction hypothesis, 
the hypothesis (ii) of Theorem~\ref{maintheorem} and Lemma~\ref{indtrick}, we can also assume that:
\begin{itemize}
\item $N_{G_\wC^\ld}(A)/A\cong(N_{G^\ld}(A)/A)_\wC$;
\item $N_{G^\ld}(A)/A$ satisfies the hypotheses of Theorem~\ref{maintheorem}.
\end{itemize}

From the induction hypothesis, applied to the group $N_{G^\ld}(A)/A$ and its solvable subgroups of order $k/|A|$, it follows that, for some 
$\bar x\in N_{G_\wC^\ld}(A)/A$ the conjugate $(H/A)^{\bar x}$ is contained in $N_{G^\ld}(A)/A$. Let $x\in N_{G_\wC^\ld}(A)$ be a lift of $\bar x$. 
The conjugate $H^x$ is then contained in $N_{G^\ld}(A)$, which proves the claim above. Let us prove that the restriction of $\iota_c$
 to conjugacy classes of solvable $\cC$-subgroups of order $k$ is also injective.
 
The same argument as in the proof of item (ii) of Theorem~\ref{pcase} shows that $H$ is subgroup $\cC$-conjugacy distinguished in $G^\ld$,
for all $\ld\in\Ld$. By Lemma~\ref{normcriterion}, this implies that $H$ is hereditarily subgroup $\cC$-conjugacy distinguished in $G$. This proves
both that $\iota_c$ is injective, when restricted to conjugacy classes of solvable $\cC$-subgroups of order $k$, and the second half of the 
second item for solvable $\cC$-subgroups of order $k$. 

The fact that the set of conjugacy classes of solvable $\cC$-subgroups of order $k$ of $G$ is finite follows, by induction, from an argument similar to the one
at the end of (ii) in Section~\ref{pcaseproof}. Let us first observe that any nontrivial solvable $\cC$-subgroup $H$ of $G$ contains a nontrivial normal 
elementary abelian $p$-subgroup $A_p$, for some $p$ such that $\Z/p\in\cC$. By the induction hypothesis, we have that 
$\cS_{f\!s}(N_G(A_p)/A_p)_{\leq k-1}$ and then $\cS_{f\!s}(N_G(A_p))_{\leq k}$ are finite. Let us then observe that, since $G$ 
is virtually $p$-torsion free, for $\Z/p\in\cC$, from (i) of Corollary~B, it follows that the set of conjugacy classes of all elementary abelian 
$p$-subgroups $A_p$ of $G$, for $\Z/p\in\cC$, is finite. These facts together imply that $\cS_{f\!s}(G)_{\leq k}$ is finite.
\smallskip

\noindent
(ii): We have already proved that solvable $\cC$-subgroups of $G$ of order $k$ are hereditarily subgroup $\cC$-conjugacy distinguished. 
Let us then prove that $\cC$-torsion elements of exponent $\leq k$ are also hereditarily $\cC$-conjugacy distinguished. 

Let $x\in G$ be a $\cC$-torsion element. We already know that the subgroup $H:=\la x\ra$ is subgroup 
$\cC$-conjugacy distinguished in $U_\ld H$, for all $\ld\in\Ld$. From the fact that the group $U_\ld H\cong U_\ld\rtimes H$ 
and its pro-$\cC$ completion $V_\ld H$ retract onto $H$, it then easily follows that $x$ is $\cC$-conjugacy distinguished 
in $U_\ld H$, for all $\ld\in\Ld$, and we conclude by Lemma~\ref{centcriterion}.
\smallskip

\noindent
(iii): By Lemma~\ref{centcriterion} and item (ii) of the theorem, for $x\in G$ a $\cC$-torsion element of exponent $\leq k$, 
we have $C_{G_\wC}(x)=\ol{C_G(x)}$. For any $\cC$-subgroup $H$ of $G$ of order $k$, we then have: 
\[C_{G_\wC}(H)=\bigcap_{x\in H}C_{G_\wC}(x)=\bigcap_{x\in H}\ol{C_G(x)}=\ol{C_G(H)}.\]
The identity $N_{G_\wC}(H)=\ol{N_G(H)}$ for a solvable $\cC$-subgroup of $G$ instead already follows from Lemma~\ref{normcriterion}
and item (ii) of the theorem.
\end{proof}

\section{Applications}
\subsection{Hyperelliptic mapping class groups of surfaces}
A \emph{marked hyperelliptic surface $(S,\cP,\u)$} is the data consisting of a closed connected oriented differentiable surface $S$, 
together with a hyperelliptic involution $\u$ on $S$ (i.e.\ a self-diffeomorphism $\u$ of $S$ such that $\u^2=\mathrm{id}_S$ 
and the quotient surface $S/\langle\u\rangle$ has genus $0$) and a finite set $\cP$ of points on $S$. We always assume
that $S\ssm\cP$ has negative Euler characteristic.

Let $\G(S)$ (resp.\ $\G(S,\cP)$) be the mapping class group of the surface $S$ (resp.\ of the marked surface $(S,\cP)$).
Let us also denote by $\u$ the image of the involution $\u$ in the mapping class group $\G(S)$ 
(note that this image is trivial if and only if $g(S)=0$ and $S$ is not hyperbolic).

The \emph{hyperelliptic mapping class group} $\U(S)$ of the hyperelliptic surface $(S,\u)$ is then defined to be the centralizer of $\u$ in $\G(S)$
and the \emph{hyperelliptic mapping class groups $\U(S,\cP)$} of the marked hyperelliptic surface  $(S,\u,\cP)$ 
is defined to be the inverse images of $\U(S)$ by the natural epimorphism $\G(S,\cP)\to\G(S)$.

If the surface $S$ has a boundary $\partial S$, let $\G(S,\dS)$ and $\G(S,\cP,\dS)$ be the associated \emph{relative} mapping class groups,
that is to say the group of relative, with respect to the subspace $\cP\cup\dS$, isotopy classes of orientation preserving 
diffeomorphisms of $S$.

The \emph{relative hyperelliptic mapping class group $\U(S,\dS)$ of the hyperelliptic surface with boundary $(S,\u,\dS)$}
is the centralizer of the hyperelliptic involution $\u$ in the relative mapping class group 
$\G(S,\dS)$ and the \emph{relative hyperelliptic mapping class group $\U(S,\cP,\dS)$} is the inverse images of $\U(S,\dS)$
via the natural epimorphism $\G(S,\cP,\dS)\to\G(S,\dS)$.
Let us observe that, for $g(S)\leq 2$, all (relative) mapping class groups are hyperelliptic in the sense of the above definitions.

Let $\wh{\U}(S,\cP,\dS)$ be the profinite completion of the relative hyperelliptic mapping class group $\U(S,\cP,\dS)$.
From the similar well-known statement for mapping class groups, it  follows immediately that the natural homomorphism
$\iota\co\U(S,\cP,\dS)\to\wh{\U}(S,\cP,\dS)$ is injective. Let us show that it also induces an isomorphism on cohomology,
that is to say that $\U(S,\cP,\dS)$ is good.

In Theorem~7.4 of \cite{Hyp}, it was proved that all hyperelliptic mapping class groups $\U(S,\cP)$ are good. The relative and standard hyperelliptic
mapping class groups are related by the short exact sequence:
\begin{equation}\label{relative}
1\to\prod_{i=1}^k\Z\to\U(S,\cP,\dS)\to\U(S,\cP)\to 1,
\end{equation}
where the free abelian group $\prod_{i=1}^k\Z$ is generated by the Dehn twists about the boundary components.
It follows from the fact that $\U(S,\cP)$ is good (cf.\ Section~2.6 in \cite{Serre}), but it is not difficult to prove directly, that the profinite completions of
this short exact sequence is also exact:
\begin{equation}\label{relativeprof}
1\to\prod_{i=1}^k\ZZ\to\wh{\U}(S,\cP,\dS)\to\wh{\U}(S,\cP)\to 1.
\end{equation}
By the Lyndon-Hochschild-Serre spectral sequence and the short exact sequences~\eqref{relative} and~\eqref{relativeprof}, we then have:

\begin{proposition}\label{Hypgood}The relative hyperelliptic mapping class group $\U(S,\cP,\dS)$ of the marked hyperelliptic surface 
with boundary $(S,\u,\cP,\dS)$ is good.
\end{proposition}

From the geometric interpretation of the groups $\U(S,\cP,\dS)$ as fundamental groups of $K(\pi,1)$ topological spaces 
(cf.\ Section~3 in \cite{Hyp}), it also follows that $\U(S,\cP)$ is of finite virtual cohomological type in the sense of Definition~\ref{fct}. 
By the short exact sequence~\eqref{relative}, the same is then true for the relative hyperelliptic mapping class group $\U(S,\cP,\dS)$.
From Theorem~\ref{pcase} and Proposition~\ref{Hypgood}, it then follows:

\begin{theorem}\label{Hyptorsion}Let $(S,\cP,\u)$ be a marked hyperelliptic surface such that $S\ssm\cP$ has negative Euler characteristic
and $p>0$ any prime number. We then have:
\begin{enumerate}
\item The natural monomorphism $\iota\co\U(S,\cP,\dS)\hookra\wh{\U}(S,\cP,\dS)$ induces a bijection of finite sets
$\iota_c\co\cS_p(\U(S,\cP,\dS))/_{\sim}\sr{\sim}{\to}\cS_p(\wh{\U}(S,\cP,\dS))/_{\sim}$. 
\item $p$-torsion elements and finite $p$-subgroups of $\U(S,\cP,\dS)$ are hereditarily (resp.\ subgroup) conjugacy distinguished.
\item For every finite $p$-subgroup $H$ of $\U(S,\cP,\dS)$, there holds: 
\[N_{\wh{\U}(S,\cP,\dS)}(H)=\ol{N_{\U(S,\cP,\dS)}(H)}\hspace{0.4cm} \mbox{and}\hspace{0.4cm} C_{\wh{\U}(S,\cP,\dS)}(H)=\ol{C_{\U(S,\cP,\dS)}(H)}.\]
\end{enumerate}
\end{theorem}

\subsection{Virtually compact special groups}
A group $G$ is \emph{compact special} if it is isomorphic to the fundamental group of a compact \emph{special cube complex}, 
whose hyperplanes satisfy certain combinatorial properties (cf.\ \cite[Sec. 3]{H-W-1}).
A group $G$ is \emph{virtually compact special}, if it contains a finite index subgroup which is compact special.
These groups were defined and studied by Daniel Wise and his collaborators and play a central role in modern geometric group theory 
(cf.\ \cite{BergeronHaglundWise2011,  H-W-1, Wise_sm-c, Wise-qc-h} and many others).

Since the seminal work of Haglund and Wise \cite{H-W-1}, many hyperbolic groups have been shown to be virtually special. 
For example, Haglund and Wise \cite{HW} prove that hyperbolic Coxeter groups are virtually compact special. 
In \cite{Wise-qc-h}, Wise proved the same for finitely generated $1$-relator groups with torsion, while
Agol \cite{Agol} showed this for fundamental groups of closed hyperbolic $3$-manifolds.
In fact, Agol \cite{Agol} proved that any hyperbolic group admitting a proper cocompact action on a CAT($0$)
cube complex is virtually compact special. Bergeron, Hagung and Wise  \cite[Theorem 1.10]{BergeronHaglundWise2011} 
proved that uniform standard arithmetic subgroups of $\SO(1,n)$ are virtually compact special. 

These groups are residually finite and hereditarily conjugacy separable (cf.\ \cite[Theorem~1.1]{MZ}). Haglund and Wise \cite{H-W-1} 
proved that any virtually compact special group $G$ contains a subgroup of finite index which is a virtual retract of a Right angled 
Artin group and so by \cite[Proposition 3.8 and the proof of Proposition 3.1]{MZ} 
$G$ is good and of finite virtual cohomological type. Hence, by Corollary~B, we have the following partial refinement of \cite[Theorem~1.1]{MZ}:





\begin{theorem}\label{virtually compact special} Let $G$ be a virtually compact special group and let $\iota\co G\hookra\wh{G}$ be the
natural homomorphism. Then, for every prime $p>0$, there holds:
\begin{enumerate}
\item $\iota$ induces a bijection of finite sets $\iota_c\co\cS_p(G)/_{\sim}\sr{\sim}{\to}\cS_p(\widehat G)/_{\sim}$.  
\item Finite $p$-subgroups  of $G$ are hereditarily subgroup conjugacy distinguished.
\item For every finite $p$-subgroup $H$ of $G$, we have the identities $N_{\widehat G}(H)=\ol{N_G(H)}$, $C_{G_\wC}(H)=\ol{C_G(H)}$.
\end{enumerate}
\end{theorem} 

If we assume, moreover, that the compact special finite index subgroup of $G$ is toral relatively hyperbolic, 
we have the following stronger result:

\begin{theorem}\label{GVCSHRNRH}Let $G$ be a group which contains a compact special and toral relatively hyperbolic subgroup of finite index
and let $\iota\co G\hookra\wh{G}$ be the natural homomorphism.
\begin{enumerate}
\item Finite subgroups  of $G$ are hereditarily subgroup conjugacy distinguished. In particular, the homomorphism
$\iota$ induces an injective map $\iota_f\co\cS_f(G)/_{\sim}\hookra\cS_f(\widehat G)/_{\sim}$.
\item For every finite subgroup $H$ of $G$, we have that $N_{\widehat G}(H)=\ol{N_G(H)}\cong\wh{N_G(H)}$ and, 
for every finitely generated subgroup $L$ of $G$, we have that $C_{\widehat G}(L)=\ol{C_G(L)}$.
\item $\iota$ induces a bijection of finite sets $\iota_s\co\cS_{f\!s}(G)/_{\sim}\sr{\sim}{\to}\cS_{f\!s}(\widehat G)/_{\sim}$.  
\end{enumerate}
\end{theorem}

\begin{proof}(i): If $G$ contains a compact special and toral relatively hyperbolic subgroup of finite index, then every finite index subgroup of $G$
has the same property. By Lemma~\ref{normcriterion}, in order to prove item (i) of the theorem, it is then enough to show that, in a group which
satisfies the hypotheses of the theorem, finite subgroups are subgroup conjugacy distinguished. 

Let then $H$ be a finite subgroup of $G$ and let us show that it is subgroup conjugacy distinguished in $G$. We will proceed by induction 
on the order $|H|$ of the subgroup. The base for the induction is provided by item (i) of Theorem~\ref{virtually compact special}, 
by which, finite subgroups of $G$ of prime power order are subgroup conjugacy distinguished. 
	
Let us assume now that $p\divides |H|$ and $|H|>p$, for a prime $p>0$. By hypothesis, $G$ is virtually torsion free.  Let then $U$ be a finite index torsion free 
normal subgroup of $G$. Since the conjugacy orbit $H^{\wh{G}}$ is a finite union of conjugates of the conjugacy orbit $H^{\wh{U}}$, in order 
to prove that $H$ is subgroup conjugacy distinguished in $G$, it is enough to prove that $H$ is subgroup conjugacy distinguished in $UH=U\rtimes H$. 

Let us then assume that $G=UH$ and that $H_1$ is a finite subgroup of $G$ such that $H^{\gamma}=H_1$ for some 
$\gamma\in \widehat U$. Let us then prove that there is a $g\in G$ with the same property. 

Let $A$ be a maximal $p$-subgroup of $H$.
By (i) of Theorem~\ref{virtually compact special}, we have that $A^{\gamma }=A^x$, for some $x\in G$. 
Since $x=uh=hu'$, for $u,u'\in U$ and $h\in H$, by the series of natural isomorphisms:
\[AU/U\cong A\widehat U/\widehat U \cong A^{\gamma} \widehat U/\widehat U= A^xU/U=A^{h}U/U,\] 
we see that the element $h$ normalizes $A$ and so $A^u=A^\gamma$. 
Hence, after replacing $H_1$ by $H_1^{u^{-1}}$ and $\gamma$ with $\gamma u^{-1}$  we may assume that 
$\gamma\in N_{\widehat U}(A)=C_{\widehat U}(A)$. 
	
Let $h\in H\ssm A$. Then $h^\gamma\in H_1\subset G$ and, since $G$ is conjugacy separable (cf.\ \cite[Theorem~1.1]{MZ}), there holds
$h^\gamma=h^g$ for some $g\in G$, so that $g\in C_{\widehat G}(h)\gamma$. 
Since $\gamma\in C_{\widehat U}(A)\subseteq C_{\widehat G}(A)$, we then have that $g\in C_{\widehat G}(h) C_{\widehat G}(A)\cap G$.

By Theorem~1.1 in \cite{MZ} and Lemma~\ref{centcriterion}, we have that $C_{\widehat G}(h)=\overline{C_G(h)}$ and, by (iii) of 
Theorem~\ref{virtually compact special}, we have that $C_{\widehat G}(A)=\overline{C_G(A)}$. Hence,
$g\in \overline{C_G(h)}\overline{C_G(A)}\cap G$. 
     	
By \cite[Theorem 1.1]{MinasyanOsin2012}, the centralizer $C_G(x)$ of any element $x\in G$ is relatively quasiconvex. 
Since the intersection of any two relatively quasiconvex subgroups is relatively quasiconvex (cf.\ \cite[Prop. 4.18]{O}), if $L$ is a finitely generated
subgroup of $G$ (in particular, if $L$ is finite), we have that $C_G(L)$ is relatively quasiconvex. 

We now apply \cite[Theorem 4.8]{GM} which states that the product of two relatively quasiconvex subgroups is closed in the profinite topology, so that
$\overline{C_G(h)}\overline{C_G(A)}\cap G= C_{G}(h)C_{ G}(A)$. We then conclude that $g=c_hc_A$ for some $c_h\in C_G(h)$ and $c_A\in C_G(A)$. 
Hence, $h^{c_A}=h^\gamma$  and then $H^{c_A}=H_1$, for $c_A\in G$,  as we had to show.
\smallskip

\noindent
(ii) From the previous item and Lemma \ref{normcriterion} follows that $N_G(H)$ is dense in $N_{\widehat G}(H)$. Let us then show that $\ol{N_G(H)}$ 
coincides with the profinite completion $\widehat{N_G(H)}$, i.e.\ that the profinite topology of $G$ induces the full profinite topology on $N_G(H)$. 
Since $C_G(H)$ is of finite index (and closed) in $N_G(H)$, it suffices to show that every finite index subgroup of $C_G(H)$ is closed in 
the profinite topology of $G$. As mentioned above, $C_G(H)$ is relatively quasiconvex and this property is closed for commensurability.
By \cite[Theorem 4.7]{GM}, relatively quasiconvex subgroups are closed in the profinite topology of $G$. Hence, every finite index subgroup 
of $C_G(H)$ (and so of $N_G(H)$) is closed in the profinite topology of $G$. This concludes the proof that  $\overline{N_G(H)}=\widehat{N_G(H)}$. 

For the statement about centralizers, it is enough to observe
that since, as we already observed, by Theorem~1.1 in \cite{MZ} and Lemma~\ref{centcriterion}, $C_{\widehat G}(x)=\overline{C_G(x)}$
for all $x\in G$. So that, for every finitely generated subgroup $L$ of $G$, with set of generators $\{x_1,\dots,x_k\}$, we have:
\[C_{\wh{G}}(L)=\bigcap_{i=1}^kC_{\wh{G}}(x_i)=\bigcap_{i=1}^k\ol{C_G(x_i)}=\ol{C_G(L)}.\]
\smallskip

\noindent
(iii): By the remarks preceding Theorem~\ref{virtually compact special}, the group $G$ is residually finite, good and of finite virtual cohomological type.
Moreover, for every  finite (in particular every finite solvable) subgroup $H$ of $G$, there is a 
natural isomorphism $\wh{N_G(H)}\cong N_{\widehat G}(H)$. Hence, to apply Corollary~D, we need to show that $N_G(H)$ satisfies 
hypothesis (i) of Corollary~D.  

Since both goodness and finiteness of virtual cohomological type are closed for commensurability (for the latter, see Lemma \ref{equivfct}), 
it suffices to prove that both properties are satisfied by $C_G(H)$. As mentioned above, $C_G(H)$ is relatively quasiconvex and so, in particular, it is a 
virtually toral relatively hyperbolic group. By \cite{D}, virtually toral relatively hyperbolic groups are of type $F$. In particular, the cohomology of $C_G(H)$ 
with coefficients in finite modules is finite. By \cite[Theorem 1]{E}, a finite index subgroup $G_0$ of $G$ has quasiconvex hierarchy and therefore 
$G_0\cap C_G(H)$ has the induced quasiconvex hierarchy, in particular this hierarchy is separable (i.e.\ closed in the profinite topology), since 
quasiconvex subgroups are separable (as we explained above). By \cite[Theorem 3.8]{GJZ}, the group $G_0\cap C_G(H)$ is then good 
and therefore $C_G(H)$ has the same property.  Thus, all hypotheses of Corollary~D are satisfied and we can conclude.
\end{proof}

\end{document}